\documentclass{amsart}
\usepackage{hyperref}
\newcommand*{\mailto}[1]{\href{mailto:#1}{\nolinkurl{#1}}}

\newtheorem{theorem}{Theorem}[section]
\newtheorem{definition}[theorem]{Definition}
\newtheorem{lemma}[theorem]{Lemma}

\newtheorem{corollary}[theorem]{Corollary}
\newtheorem{remark}[theorem]{Remark}

\newcommand{\R}{{\mathbb R}}
\newcommand{\N}{{\mathbb N}}

\newcommand{\C}{{\mathbb C}}

\newcommand{\nn}{\nonumber}
\newcommand{\be}{\begin{equation}}
\newcommand{\ee}{\end{equation}}

\newcommand{\ti}{\tilde}

\newcommand{\id}{{\rm 1\hspace{-0.6ex}l}}
\newcommand{\E}{\mathrm{e}}
\newcommand{\I}{\mathrm{i}}

\newcommand{\im}{\mathrm{Im}}
\newcommand{\re}{\mathrm{Re}}

\newcommand{\floor}[1]{\lfloor#1 \rfloor}

\newcommand{\eps}{\varepsilon}

\newcommand{\sig}{\sigma}
\newcommand{\lam}{\lambda}
\newcommand{\gam}{\gamma}

\newcommand{\gk}{\kappa}

\numberwithin{equation}{section}

\begin{document}

\title[Spectral Asymptotics for Schr\"odinger Operators]{Spectral Asymptotics for Perturbed Spherical Schr\"odinger Operators and Applications to Quantum Scattering}

\author[A.\ Kostenko]{Aleksey Kostenko}
\address{Faculty of Mathematics\\ University of Vienna\\
Nordbergstrasse 15\\ 1090 Wien\\ Austria}
\email{\mailto{duzer80@gmail.com};\mailto{Oleksiy.Kostenko@univie.ac.at}}

\author[G.\ Teschl]{Gerald Teschl}
\address{Faculty of Mathematics\\ University of Vienna\\
Nordbergstrasse 15\\ 1090 Wien\\ Austria\\ and International
Erwin Schr\"odinger
Institute for Mathematical Physics\\ Boltzmanngasse 9\\ 1090 Wien\\ Austria}
\email{\mailto{Gerald.Teschl@univie.ac.at}}
\urladdr{\url{http://www.mat.univie.ac.at/~gerald/}}

\dedicatory{Dedicated with great pleasure to Israel Samoilovich Kac on the occasion of his 85th birthday.}
\thanks{{\it Research supported by the Austrian Science Fund (FWF) under Grant No.\ Y330 and No.\ M1309}}
\thanks{Commun. Math. Phys. {\bf 322}, 255--275 (2013)}


\begin{abstract}
We find the high energy asymptotics for the singular Weyl--Titchmarsh $m$-functions and the associated spectral measures
of perturbed spherical Schr\"o\-dinger operators (also known as Bessel operators). We apply this result to establish
an improved local Borg--Marchenko theorem for Bessel operators as well as uniqueness theorems for the radial
quantum scattering problem with nontrivial angular momentum.
\end{abstract}

\maketitle

\section{Introduction}

In this paper we will investigate perturbed spherical Schr\"odinger operators (also known as Bessel operators)
\be\label{eq:bes}
\tau=-\frac{d^2}{dx^2}+\frac{l(l+1)}{x^2}+q(x),\quad l\ge -\frac{1}{2},\qquad x\in\R_+:=(0,+\infty),
\ee
where the potential $q$ is real-valued satisfying
\begin{equation}\label{q:hyp}
q \in L^1_{\mathrm{loc}}(\R_+), \qquad \begin{cases} x\, q(x) \in L^1(0,1), & l>-\frac{1}{2},\\
x(1-\log(x)) q(x) \in L^1(0,1), & l = -\frac{1}{2}. \end{cases}
\end{equation}
Note that we explicitly allow non-integer values of $l$ such that we also cover the case of
arbitrary space dimension $n\ge 2$, where $l(l+1)$ has to be replaced by $l(l+n-2) + (n-1)(n-3)/4$ \cite[Sect.~17.F]{wdln},
or the case of scattering of waves and particles in conical domains \cite{che}.
Due to its physical importance this equation has obtained much attention in the past
and we refer for example to \cite{ahm}, \cite{ahm2}, \cite{gr}, \cite{kst}, \cite{lv}, \cite{se}, \cite{wdln} and the references therein.

Of course one of the most interesting applications is the scattering problem in $\R^3$ with a spherically symmetric potential.
Originating in the seminal work of Heisenberg, the question if the potential $q$ is determined
by the scattering phase for one fixed value of the angular moment $l$ has a venerable history.
 In the simplest case $l=0$ this question is completely understood by now and we refer, e.g, to the monographs \cite{cs}, \cite{fad}, or \cite{Mar11}.
However, for general $l$ this question remained open to the best of our knowledge. As one of our main results we will show
that, in combination with the eigenvalues and norming constants, the scattering phase
determines the potential uniquely. 

Our approach will use singular Weyl--Titchmarsh--Kodaira theory as one basic ingredient. This approach was first used by Kodaira \cite{ko} in connection with radial
scattering theory as already mentioned above. This original work contained some gaps first pointed out by Kac \cite{ka} who proposed an alternate approach based
on Krein's method of directing functionals. However, these results did not
get much attention until recently when Gesztesy and Zinchenko \cite{gz} took it up again and triggered a large amount of results
\cite{eck}, \cite{et2}, \cite{ful08}, \cite{fl}, \cite{fll}, \cite{kst}, \cite{kst2}, \cite{kst3}, \cite{kt}.

Let us outline the content of our paper. We will first establish the high energy asymptotics of the singular Weyl function and the spectral measure associated
with operators of the type \eqref{eq:bes} (see Theorem \ref{th:main}). As one expects, the leading term in the asymptotic behavior is given by the (known)
unperturbed Weyl function. However, while the form of this result suggests it can be proven by a straightforward perturbation argument, this is not the case.
In fact, the definition of the Weyl function depends on a basis of entire (with respect to the spectral parameter $z$) solutions $\phi(z,x)$ and $\theta(z,x)$ of
the underlying differential equation. But while the {\em regular} solution $\phi(z,x)$ can be obtained using the usual iteration scheme, this fails for the
{\em singular} solution $\theta(z,x)$ (unless additional stringent assumptions on $q(x)$ are made --- see also Remark~\ref{rem:2.01}). This is precisely
the fact originally overlooked by Kodaira \cite{ko} and first pointed out by Kac \cite{ka} (who conjectured that such a solution will not exist at all
in general -- see his footnote on p. 206). Moreover, there have also been some additional incorrect constructions and we refer to \cite{kst2}, \cite{kt} for
further details. Consequently, the usual strategy for proving high energy asymptotics of the Weyl function are not available in this situation.
Hence we use an entirely different approach based on two main ingredients: The one-term asymptotic formula for the $m$-function
(Theorem \ref{th:ben}) and the commutation methods for Bessel operators \cite{kst3}. The proof will be given in Section~\ref{sec:krein}, where
we will use a Liouville type transform to establish a connection between the perturbed Bessel operators and the Krein string operators $-\frac{d^2}{r(\xi)d\xi^2}$.

In the remaining sections we will then demonstrate the usefulness of our main result.
In Section~\ref{sec:BM}, we apply Theorem~\ref{th:main} to derive an improved local Borg--Marchenko result
for Bessel operators (including some applications to inverse spectral uniqueness results). 
In the final section, we will apply our finding to the quantum mechanical scattering problem.
More precisely, using the connection between the Jost function \cite{jo47} and the singular $m$-function and then applying our local Borg--Marchenko result (Theorem~\ref{th:uniq}),
we are able to establish uniqueness in terms of both the absolute value of the Jost function and the phase shift (see Theorems \ref{th:4.1} and \ref{th:5.3}).
Let us also mention that basic properties of the Jost function in the case of general $l$ were first formulated in \cite{ko}. Asymptotics for the Jost solutions and the Jost functions
are well-known and were investigated by several authors \cite{cc}, \cite{gt}, \cite{gpt}. Moreover, a scattering approach between the unperturbed
operators for $l=0$ and $l\in[-\frac{1}{2},\frac{1}{2})$ was recently introduced in \cite{ahm2}.

Finally, necessary results on the asymptotic behavior of $m$-functions of regular Sturm--Liouville problems are collected in Appendix A.
In Appendix~\ref{app:B}, we provide some estimates and analytical properties of the solutions of $\tau y=z\, y$.

\section{Asymptotics of the Weyl function}

Let $(0,b) \subseteq \R_+$ be some open interval.
We will use $\tau$ to describe the formal differential expression and $H$ to describe
the self-adjoint operator acting in $L^2(0,b)$ and given by $\tau$ together with the usual boundary condition at $x=0$:
\be\label{eq:bc_dir}
\lim_{x\to0} x^l ( (l+1)f(x) - x f'(x))=0, \qquad l\in[-\frac{1}{2},\frac{1}{2}).
\ee
We are mainly interested in the case where $\tau$ is limit point at $b$, but if it is not, we simply
choose another boundary condition there.

Next, notice that our assumption \eqref{q:hyp} implies that the equation 
\be\label{eq:sp01}
\tau y = z y 
\ee
has a system of solutions $\phi(z,x)$ and $\theta(z,x)$ which is real entire with respect to $z$ such that 
\be\label{eq:fs01}
\phi(z,x)=x^{l+1}\ti{\phi}(z,x),\quad  
\theta(z,x)=
\begin{cases} \frac{x^{-l}}{2l+1}\ti{\theta}(z,x), & l>-\frac{1}{2},\\
-x^{1/2}\log(x)\ti{\theta}(z,x) , & l = -\frac{1}{2}, \end{cases},
\ee
 where $\ti{\phi},\ti{\theta}\in W^{1,1}[0,1]$ and  $\ti{\phi}(0)=\ti{\theta}(0)=1$. For a detailed construction of these solutions we refer to, e.g., \cite{kt}.
 {\em The singular Weyl function} $m:\C\setminus\R\to \C$ is defined such that 
 \be\label{eq:s_m}
\psi(z,x)= \theta(z,x) + m(z) \phi(z,x),\quad \lambda\in\C\setminus\R, 
\ee
either belongs to $L^2(c,b)$ for some (and hence for all) $c\in(0,b)$ if $H$ is limit point at $b$, or satisfies the boundary condition at $x=b$ if $H$ is limit circle at $b$.
Note that, while the first solution $\phi(z,x)$ is unique under this normalization, the second solution $\theta(z,x)$ is not, since for any real entire $f(z)$ the new solution
$\ti{\theta}(z,x)=\theta(z,x)-f(z)\phi(z,x)$ also satisfies \eqref{eq:fs01}. Note that the corresponding singular $m$-function $\ti{m}$ is given by
\be
\ti{m}(z)=m(z)+f(z)
\ee
in this case.
Moreover, it was shown in \cite{kst}, \cite{kt} that the singular $m$-function \eqref{eq:s_m} admits the following integral representation 
\be\label{eq:int_rep}
m(z)=g(z)+(1+z^2)^{\kappa_l}\int_{\R}\Big(\frac{1}{\lam-z}-\frac{\lam}{1+\lam^2}\Big)\frac{d\rho(\lam)}{(1+\lam^2)^{\kappa_l}},\quad z\notin\R.
\ee
Here $\kappa_l:=\floor{\frac{l}{2}+\frac{3}{4}}$, the function $g$ is real entire, and $\rho:\R\to\R$ is a nondecreasing function satisfying
\be\label{eq:sm_norm}
\rho(\lam)=\frac{\rho(\lam+)+\rho(\lam-)}{2},\quad \rho(0)=0,\quad \int_{\R}\frac{d\rho(\lam)}{(1+\lam^2)^{\kappa_l+1}}<\infty.
\ee
The operator $H$ is unitarily equivalent to multiplication by the independent variable in $L^2(\R,d\rho)$ and thus
$\rho$ is called {\em the spectral function} and $d\rho$ is {\em the spectral measure}. Moreover, $m$ can be identified as a $Q$
function, in the sense of Krein, in the framework of super-singular perturbations. This is explained in detail in \cite{kt} (see also \cite{DSh_00}, \cite{kl} and \cite{Der98} and references therein). We also remark that the value of
$\kappa_l$ is best possible (see again \cite{kt} and also \cite{kl}, where the case $q(x)=\frac{a}{x}$, $a\in\R$, was treated).

In the special case $l=0$ the function $m$ is the classical $m$-function and Marchenko \cite{Mar52} proved that in any nonreal sector in $\C_+$ the $m$-function satisfies
\be\label{eq:mar}
m(z)= -\sqrt{-z}(1+o(1)),\quad |z|\to +\infty,
\ee
where the branch cut of the root is taken along the negative real axis.
The estimate for the remainder term was later improved  by Krein, Levitan, and Marchenko (see, e.g., \cite{Mar11}, \cite{tschroe}).
By now there is a vast literature on high energy asymptotics for $m$-functions of general Sturm--Liouville operators and it
seems the first results were obtained by M.G. Krein and I.S. Kac \cite{K71, K73, KK2}, Hille \cite{hil63},  Everitt \cite{ev72} and Kasahara \cite{kas}.
The most complete results on one-term asymptotics for the $m$-functions were obtained by Bennewitz \cite{Ben} (cf.\ also our Appendix \ref{app:A}).  

In the present paper we are interested in the asymptotic behavior of the singular $m$-function \eqref{eq:s_m} of perturbed Bessel operators. 
Our main result is the following extension of Marchenko's asymptotic formula \eqref{eq:mar}.
\begin{theorem}\label{th:main}
Suppose
\be
H = -\frac{d^2}{dx^2} + \frac{l(l+1)}{x^2} + q(x),
\ee
where $q$ satisfies \eqref{q:hyp}.
Let $m(\cdot)$ be the singular $m$-function \eqref{eq:s_m}. Then there is a real entire function $g$ such that in any nonreal sector,
\be\label{eq:2.10}
m(z)-g(z)=m_l(z)(1+o(1)),\quad |z|\to+\infty,
\ee
where
\be\label{eq:II.11}
m_l(z) = \begin{cases}
\frac{-C_l^2}{\sin((l+\frac{1}{2})\pi)} (-z)^{l+\frac{1}{2}}, & {l+\frac{1}{2}}\in\R_+\setminus \N_0,\\
\frac{-C_l^2}{\pi} z^{l+\frac{1}{2}}\log(-z), & {l+\frac{1}{2}} \in\N_0,\end{cases}\quad C_l=\frac{\sqrt{\pi}}{\Gamma(l+\frac{3}{2}) 2^{l+1}}.
\ee
Moreover, the spectral function satisfies
\be
\rho(\lam)=\rho_l(\lam)(1+o(1)),\quad \lam\to +\infty,
\ee
where
\be
\rho_l(\lam) = \frac{C_l^2}{\pi(l+\frac{3}{2})}\chi_{[0,\infty)}(\lam) \lam^{l+\frac{3}{2}},  \qquad l \geq -\frac{1}{2}.
\ee
\end{theorem}   

\begin{remark}
We can always choose a singular solution $\theta(z,x)$ to be a Frobenius type solution (for a definition and basic properties we refer to \cite{kt}). Note that in this case the function $g(.)$ in \eqref{eq:int_rep} and hence in \eqref{eq:2.10} becomes a real polynomial of order no greater than $2\kappa_l+1$ (see \cite[Thm.~4.5]{kt}). 
\end{remark}


\section{Proof of Theorem~\ref{th:main}}\label{sec:krein}

First, let us note that it suffices to prove Theorem \ref{th:main} in the case when $b$ is a regular endpoint. In fact,
by \cite[Lem.~7.1]{kst}, we know
\be
m(z)=-\frac{\theta(z,c)}{\phi(z,c)}+O(\frac{1}{\sqrt{-z}\phi^2(z,c)}),
\ee
as $|z|\to +\infty$ in any nonreal sector. Hence the asymptotic behavior of $m$ depends only on the behavior of $q$ near
$0$ and is independent of the behavior of $q$ outside a neighborhood of $0$ as well as of a possible boundary condition
at $b$.

Therefore, without loss of generality we set $b=1$ and assume that $q\in L^1(c,1)$ for some $c\in (0,1)$. We divide the proof into three steps.
First we establish a connection with the Krein string operator in the case $l\in[-\frac{1}{2},\frac{1}{2})$. Second we use this to prove our
main result in this case. Third we extend this result to all $l$ using commutation methods.

\subsection{Connection with the Krein string operator}\label{ss:krein}
In this section we restrict our considerations to the case when $x=0$ is in the limit circle case, that is $l\in[-\frac{1}{2},\frac{1}{2})$.
It is well-known that under this assumption the self-adjoint operator $H$ associated with $\tau$ is lower semibounded.
We continue to use the same basis of solutions $\phi(z,x)$ and $\theta(z,x)$ as in the previous section.
However, instead of the boundary condition \eqref{eq:bc_dir} induced by $\phi$ we will now use the one induced by $\theta(\lam_0)$
plus an arbitrary boundary condition at $x=1$,
\be\label{eq:bc}
\lim_{x\to0} W_x(f,\theta(\lam_0))=0,\quad \cos(\beta) f(1)- \sin(\beta) f'(1)=0.
\ee
Here $\lam_0\in\R$ and $\beta\in [0,\pi)$ are fixed, and we can assume that $\theta(z,x)$ satisfies $W(\theta(\lam_0),\theta(z))=0$ (cf. \cite[App.~A]{kst2}). In this case $\theta(z,x)$ continues to satisfy \eqref{eq:fs01}.

The singular Weyl function is then defined such that
\be\label{eq:s_mN}
\psi(z,x)= \phi(z,x) - \ti{m}(z) \theta(z,x)
\ee
satisfies the boundary condition at $x=1$. In this case $\ti{m}$ becomes the classical $m$-function (see, e.g., \cite[App.~A]{kst2}) and hence it will be a Herglotz--Nevanlinna function satisfying
\be
\ti{m}(z) = \re(\ti{m}(\I)) + \int_\R \left(\frac{1}{\lam-z} - \frac{\lam}{1+\lam^2}\right) d\rho(\lam),
\ee
where the spectral measure $\ti{\rho}$ satisfies
$\int_\R d\ti{\rho}(\lam)=\infty$ and $\int_\R \frac{d\ti{\rho}(\lam)}{1+\lam^2}<\infty$.

Let $r\in L^1(0,a)$ be a positive function on $(0,a)$ and consider the Krein string operator
\be
L=-\frac{1}{r(\xi)}\frac{d^2}{d\xi^2}
\ee
acting in the weighted Hilbert space $L^2_{r}(0,a)$ and subject to the boundary conditions
\be\label{eq:bc_b}
y'(0)= \cos(\tilde{\beta}) y(a) - \sin(\tilde{\beta}) y'(a) =0.
\ee
Let also $c(z,\xi)$ and $s(z,\xi)$ be the fundamental solutions of 
\be\label{eq:sp02}
-y'' =z\, r(\xi)y,\quad \xi\in [0,a],
\ee
such that
 \be\label{eq:fs02}
c(z,0)=s'(z,0)=1,\quad c'(z,0)=s(z,0)=0.
\ee
Define the corresponding $m$-function $M:\C\setminus\R\to \C$ as follows:
\be\label{eq:Mb}
y_\beta(z,\xi):=s(z,\xi)-M(z)c(z,\xi)\quad \text{satisfies \eqref{eq:bc_b} at $\xi=a$}.
\ee

The main aim of this section is to establish a connection between the operators $L$ and $H$ and hence between the $m$-coefficients $\ti{m}$ and $M$. 

Without loss of generality we can assume that $H\ge \eps\id$ with some $\eps>0$. Hence the solution $\theta_0(x):=\theta(0,x)$ is positive on $(0,1]$. Then 
 (see, e.g., \cite[\S 14]{KK2}), we set
 \be\label{eq:r}
 \xi:=\xi(x)=\int_0^x \frac{dt}{\theta_0^2(t)},\quad a:=\xi(1),\quad r(\xi):=\theta_0^4(x),
 \ee 
and define the map $U:L^2(0,1)\to L^2_{r}(0,a)$ as follows:
\be\label{eq:U}
U:v(x)\to u(\xi):= \frac{1}{\theta_0(x)}v(x).
\ee
First of all, notice that $r\in L^1(0,a)$. Indeed, using \eqref{eq:fs01}, we get
\[
\int_0^a r(\xi)d\xi=\int_0^1\theta_0^4(x)\frac{dx}{\theta_0^2(x)}=\int_0^1\theta_0^2(x)dx<\infty.
\]
Hence the above definition is correct. Moreover, $U$ is isometric:
\[
\|Uv\|^2_{L^2_{r}}=\int_0^a|u(\xi)|^2r(\xi)d\xi=\int_0^1 \Big| \frac{v(x)}{\theta_0(x)}\Big|^2 \theta_0^4(x) \frac{dx}{\theta_0^2(x)}=\int_0^1|v|^2dx=\|v\|^2_{L^2}.
\]
Furthermore, it is not difficult to check that $\ti{y}:=Uy$ solves \eqref{eq:sp02} if $y$ is a solution of \eqref{eq:sp01}. Indeed, this is immediate from the following representation of \eqref{eq:sp01} (cf. \cite[\S 14]{KK2})
\[
-\theta_0^2(x)\frac{d}{dx}\Big(\theta_0^2(x)\frac{d}{dx}\frac{y}{\theta_0(x)}\Big)=z\, \theta_0^4(x)\, \frac{y}{\theta_0(x)}.
\]
Finally, let us show that $c(z,\xi)=U\theta(z)$ and $s(z,\xi)=U\phi(z)$. Notice that
\[
(Uy)(0)=\lim_{x\to 0}\frac{y(z,x)}{\theta_0(x)}
=\begin{cases}
(2l+1)\lim_{x\to 0}x^{l}y(z,x),& l\in(-\frac{1}{2},\frac{1}{2}),\\
-\lim_{x\to 0}\frac{y(z,x)}{\sqrt{x}\log(x)},& l=-\frac{1}{2},
\end{cases}
\]
and 
\[
\frac{d}{d\xi}Uy\Big|_{\xi=0}=\lim_{x\to 0}\theta_0^2(x)\frac{d}{dx}\frac{y}{\theta_0(x)}=\lim_{x\to 0}\big(y'\theta_0-y\theta_0'\big)=\lim_{x\to 0}W_x(y,\theta_0).
\]
Hence, using the representation of $\phi$ and $\theta$ from \cite[\S 3]{kt}, we easily compute that
\[
(U\phi)(0)=0,\quad (U\phi)'_{\xi}(0)= 1,
\]
and 
\[
(U\theta)(0)= 1,\quad (U\theta)'_{\xi}(0)=0.
\]
Since both $U\phi$ and $U\theta$ solve \eqref{eq:sp02}, we are done. Similarly one computes
\be\label{eq:betas}
\cos(\ti{\beta}) (U\psi)(a)- \sin(\ti{\beta}) (U\psi)'(a)=0, \qquad \cot(\ti{\beta}) = \theta_0(1) \big(\cot(\beta)+1 \big).
\ee

Thus we proved the following result.

\begin{lemma}\label{lem:01}
Let $l\in [-1/2,1/2)$ and suppose \eqref{q:hyp}. Let also $a$,  $r(\cdot)$, and the unitary map $U$ be defined by \eqref{eq:r}--\eqref{eq:U}. If $\beta$ and $\ti\beta$ are connected by \eqref{eq:betas}, then the operators $H$ and $L$ are unitarily equivalent: $H=U^{-1}LU$. Moreover, the fundamental solutions \eqref{eq:fs01} and \eqref{eq:fs02} are connected by
\[
c(z,\xi)=\big(U\phi(z)\big)(\xi),\quad s(z,\xi)=\big(U\theta(z)\big)(\xi),
\]
and the corresponding $m$-functions defined by \eqref{eq:s_mN} and \eqref{eq:Mb} satisfy
\be
\ti{m}(z)=M(z),\quad z\notin\R.
\ee
\end{lemma}
\begin{corollary}\label{cor:2.01}
Let $l\in [-1/2,1/2)$ and let $H_{(0,c)}$ be a self-adjoint operator associated with \eqref{eq:bes} on $(0,c)$ and subject to some separated
boundary conditions at the endpoints $0$ and $c$. Then the eigenvalues $\lambda_n(c)$ of $H_{(0,c)}$ satisfy 
\be\label{eq:eg_as}
\lim_{n\to \infty}\frac{n}{\sqrt{\lam_n(c)}}=\frac{c}{\pi}.
\ee 
\end{corollary}
\begin{proof}
Due to the interlacing properties of eigenvalues, it suffices to prove the claim in the case of boundary conditions \eqref{eq:bc}.
By Lemma \ref{lem:01}, the spectra of operators $H_{(0,c)}$ and $L_{(0,\xi(c))}$ coincide. However, by \cite[\S 11.8]{KK2}, the eigenvalues $\lam_n(\xi)\, (=\lam_n(c))$ of $L$ satisfy
\[
\lim_{n\to \infty}\frac{n}{\sqrt{\lam_n(\xi)}}=\frac{1}{\pi}\int_0^{\xi(c)} \sqrt{r(t)}dt=\frac{1}{\pi}\int_0^c \sqrt{\theta_0^4(t)}\frac{dt}{\theta_0^2(t)}=\frac{c}{\pi}.
\]
\end{proof}
\begin{remark}\label{rem:2.01}
(i) Using commutation methods (see, e.g., \cite{kst3}), \eqref{eq:eg_as} can be extended to the case $l\ge -1/2$.
However, these asymptotics were already obtained in \cite{kst} (see also \cite{ahm}) together with a detailed estimate for the
error term and hence there is no need to further pursue this approach here.

(ii) Note that \eqref{eq:eg_as} indicates that the solutions $\phi(z,x)$ and $\theta(z,x)$ are entire functions in $z$ of order $1/2$ and of finite type. Indeed, this has been shown in
\cite[Lem.~8.4]{kst2}).

Moreover, there is a standard iteration procedure to construct the solution $c(z,x)$ of \eqref{eq:sp02} (see, e.g., \cite[\S 2]{KK2}) and then, by using the Liouville transform, we obtain  the solution $\theta(z,x)$. As already mentioned in the introduction, it is difficult to construct $\theta(z,x)$ directly by iteration. Indeed, for the standard iteration scheme (see, e.g., \cite[Lem.~2.2]{kst} for a construction of $\phi(z,x)$) to converge it is required that $q$ must satisfy an additional assumption at $x=0$. For instance, $q\in L^1(0,\eps)$ is required if $l=0$. Finally, it is difficult to use a perturbative approach to get a detailed asymptotic for $c(z,x)$ and hence for $\theta(z,x)$ since \eqref{eq:sp02} needs to be considered as a perturbation of $-y''=z\xi^{\alpha}y$ and thus changes the coefficient containing the spectral parameter $z$.
\end{remark}

\subsection{Spectral asymptotics in the case $l\in [-1/2,1/2)$} \label{ss:krein_asymp}

Using the connection between the $m$-functions established in Lemma \ref{lem:01}, we are able to prove the following result:

\begin{theorem}\label{th:reg}
Let $l\in [-1/2,1/2)$, suppose $q$ satisfies \eqref{q:hyp}, and let $H$ be given by \eqref{eq:bes} together with the boundary condition \eqref{eq:bc}. Then the corresponding $m$-function \eqref{eq:s_mN} satisfies 
\be\label{eq:m_as}
\ti{m}(z)=-\frac{1}{m_l(z)}(1+o(1))\quad \text{as} \quad |z|\to \infty,
\ee
where $m_l$ is given by \eqref{eq:II.11}. 
The latter holds uniformly in any nonreal sector in $\C_+$. 
\end{theorem}

\begin{proof}
{\em (i) The case $l\in(-1/2,1/2)$.}
By \eqref{eq:fs01} and Lemma A.2 from \cite{kt}, we get 
\[
\xi(x)=(2l+1)x^{2l+1}\ti{\xi}(x),\quad
 \ti{\xi}\in W^{1,1}[0,1],\quad \ti{\xi}(0)=1.
\]
Hence, $\xi$ as a function of $x$ has limit order $2l+1$ and by Lemma \ref{lem:a.asymp}, the inverse $x=x(\xi)$ of $\xi(x)$ has  the following asymptotic behavior:
\[
x(\xi)=\Big(\frac{\xi}{2l+1}\Big)^{\frac{1}{2l+1}}(1+o(1))\quad \text{as}\quad \xi\to 0.
\]
Next, using \cite[Lem.~A.2]{kt}, we obtain
\[
R(\xi)=\int_0^\xi r(\xi)d\xi=\int_0^x\theta_0^2(t)dt=\frac{x^{1-2l}}{(2l+1)^2(1-2l)}(1+o(1))=\frac{\xi^{\alpha}}{A_\alpha}(1+o(1)),
\]
where
\[
\alpha=\frac{1-2l}{1+2l},\quad
A_\alpha:=(1-2l)(1+2l)^{\frac{2l+3}{2l+1}}.
\]
Therefore, by Theorem \ref{th:ben},  the $m$-function \eqref{eq:Mb} corresponding to the operator $L$ satisfies
\[
M(z)=K_{\frac{1}{1+\alpha}}A_{\alpha}^{\frac{1}{1+\alpha}} (-z)^{-\frac{1}{1+\alpha}}(1+o(1))\quad \text{as} \quad |z|\to \infty; \quad K_\nu=\frac{\nu^{1-\nu}\Gamma(\nu)}{(1-\nu)^\nu \Gamma(1-\nu)}.
\]
The latter holds uniformly in any nonreal sector in $\C_+$. Further, noting that
\[
\frac{1}{1+\alpha}=l+\frac{1}{2},
\]
after straightforward calculations we get
\be\label{eq:k_l}
K_{l+\frac{1}{2}}A_{\alpha}^{l+\frac{1}{2}}=2^{2l}(2l+1)^2\frac{\Gamma(1/2+l)}{\Gamma(1/2-l)}=\frac{\sin(\pi(l+\frac{1}{2}))}{C_l^2}
\ee
and Lemma \ref{lem:01} completes the proof.

{\em (ii) The case $l=-\frac{1}{2}$.}
By \eqref{eq:fs01}, $\theta_0(x)=\sqrt{x}\log(x)\tilde{\theta}_0(x)$, where $\tilde{\theta}_0\in W^{1,1}[0,1]$ and $\tilde{\theta}_0(0)=1$. Hence by \eqref{eq:r} we get
\be\label{eq:xi12}
\xi=G(x):=\int_0^x \frac{dt}{t\log^2(t) \ti\theta_0^2(t)},\quad R(\xi)=P(x):=\int_0^x t\log^2(t) \ti\theta_0^2(t)dt.
\ee
Set
\[
G_0(x):=\int_0^x \frac{dt}{t\log^2(t) }=\frac{-1}{\log(x)},\:\:
P_0(x):=\int_0^x t\log^2(t) dt=\frac{x^2}{2}(\log^2(x)-\log(x)+\frac{1}{2}).
\]
Note that $P$, $P_0$, $G$, and $G_0$ are absolutely continuous and strictly increasing. Let $p$, $p_0$, $g$, and $g_0$ be the corresponding inverses.
 
Firstly, since $\tilde{\theta_0}\in W^{1,1}$ and $\tilde{\theta_0}(0)=1$, we observe that
\be\label{eq:g_as}
(G\circ g_0)(x)=\int_0^{g_0(x)}\frac{dt}{t\log^2(t) \ti\theta_0^2(t)} =\int_0^x \frac{dt}{(\ti\theta_0^2 \circ g_0)(t)}\sim x,\quad x\to 0.
\ee
Furthermore, by L'H\^opital's rule, $P\sim P_0$ and $G\sim G_0$ as $x\to 0$. Moreover,
\[
\frac{P_0(x+o(x))}{P_0(x)}=(1+o(1))^2\frac{\log^2(x)-\log(x)(1+o(1))+1/2+o(1)}{\log^2(x)-\log(x)+1/2}\sim 1, \quad x\to 0,
\] 
and hence 
\be\label{eq:f_as}
P(x+o(x))\sim P(x),\quad x\to 0.
\ee
Noting that $R(\xi)=(P\circ g)(\xi)$ and $R_0(\xi)=(P_0\circ g_0)(\xi)$ and using \eqref{eq:g_as} and \eqref{eq:f_as}, we get
\be\label{eq:R_as}
\lim_{\xi\to 0}\frac{R(\xi)}{R_0(\xi)}=\lim_{\xi\to 0}\frac{(P\circ g)(\xi)}{(P_0\circ g_0)(\xi)}=\lim_{x\to 0}\frac{(P\circ g\circ G_0)(x)}{P_0(x)}=1.
\ee
The function $R_0$, given by 
\be
R_0(\xi)=(P_0\circ g_0)(\xi)=\frac{2+2\xi+\xi^2}{4\xi^2\E^{2/\xi}}\sim \ti{R}_0(\xi):=\frac{1}{2\xi^2\E^{2/\xi}},\quad \xi\to 0,
\ee
has limit order $\infty$ at $0$ and hence, by \eqref{eq:R_as}, so has $R$. 
By Theorem \ref{th:ben}, the asymptotic behavior of $m(z)$ as $z\to\infty$ is determined by the function $f=F^{-1}$, where $F(\xi):=\frac{1}{ \xi R(\xi)}$. 
By Lemma  \ref{lem:a.asymp},
the inverse function $f=F^{-1}$ is asymptotically equal to the inverse $\ti{f}_0$ of 
\[
\ti{F}_0(\xi):=\frac{1}{ \xi \ti{R}_0(\xi)} =2\xi\E^{2/\xi}.
\]
Note that the function $\ti{F}_0(\xi)$ is a bijection from $(0,2)$ to $(4\E,\infty)$ whose inverse is given by
\[
\ti{f}_0(x)=-\frac{2}{W_m(-\frac{4}{x})},
\]
where $W_m$ is the second branch of the Lambert W-function \cite[\S4.13]{dlmf} which satisfies
\be\label{eq:w_as}
-W_m(-\frac{1}{x})=\log(x)+\log(\log(x))+O\big(\frac{\log(\log(x))}{\log(x)}\big), \quad x\to +\infty
\ee
Finally, applying Theorem \ref{th:ben}, we  get
\[
m(\rho\E^{\I \varphi})=\ti{f}_0(\rho)(1+o(1))=\frac{2}{\log(\rho)}(1+o(1)),\quad \rho\to \infty.
\]
Lemma \ref{lem:01} completes the proof. 
\end{proof}

Observe that the $m$-functions \eqref{eq:s_m} and \eqref{eq:s_mN} are connected by $m(z)=-1/\ti{m}(z)$. 
Therefore, the asymptotic formula \eqref{eq:m_as} proves Theorem \ref{th:main} in the case $l\in [-1/2,1/2)$. 

\subsection{Spectral asymptotics in the case $l\ge\frac{1}{2}$.}

Finally the general case can be reduced to the previous one as in the proof of Corollary~3.11 from \cite{kst3}.
In fact, given $l=l_0+k$ with $l_0\in[-\frac{1}{2},\frac{1}{2})$, $k\in\N$, we can perform a single commutation step to
obtain an operator $\check{H}$ associated with $l-1$ and a new potential $\check{q}$ in the same class.
Now taking into account that $\check{\phi}(z,x)= (2l+1) x^l (1+o(1))$ does not satisfy our normalization \eqref{eq:fs01}
\cite[Thm.~3.7]{kst3} implies $m(z) = (2l+1)^2 (z-\lam) \check{m}(z)$ and the claim follows by induction on $k$.

\begin{remark}\label{rem:2.02}
As pointed out in our introduction the solution $\theta(z,x)$ is only unique up to a term $f(z) \phi(z,x)$.
Following \cite{kt} we will call $\theta(z,x)$ {\em a Frobenius solution} if
\be
\lim_{x\to0}x^{-(l+1)}\frac{\partial^{(n_l+1)}}{\partial z^{(n_l+1)}}\theta(z,x)\equiv 0,
\ee
where $n_l:=\floor{l+1/2}$. This will fix $f(z)$ up to a polynomial of degree $n_l$ and ensure that
$m(z)$ is in the generalized Nevalinna class $N^\infty_{\gk_l}$, where $\gk_l=\floor{\frac{l}{2}+\frac{3}{4}}$ (we refer to \cite{kt} for further details).

It is easy to see that $\theta$ in Section \ref{ss:krein} is a Frobenius solution. Moreover, the  property of a singular solution $\theta$ to be Frobenius is invariant under a single commutation. Therefore, it follows from Remark \ref{rem:2.01} that a Frobenius solution $\theta$ is an entire function in $z$ of growth order $1/2$ and of finite type. 
\end{remark}


\section{Uniqueness results for the inverse spectral problem}\label{sec:BM}

As it was discussed in Remark \ref{rem:2.02}, if we choose $\theta(z,x)$ to be a Frobenius solution, then the $m(z)$ is in the generalized Nevalinna class $N^\infty_{\gk_l}$. This fixes $m(z)$ up to a polynomial of degree $n_l$
and thus $l$ can be read off from the asymptotics of $m(z)$ in this case.

This observation enables us to show the following improvement of Theorem 8.5 from \cite{kst2}.
\begin{theorem}\label{th:uniq}
Let $H(l_j,q_j)$, $j\in \{1,2\}$, be the operators defined in $L^2(0,b_j)$ by \eqref{eq:bes}. Assume that $q_j$ satisfies \eqref{q:hyp}, $l_j\in[-\frac{1}{2},\infty)$, and $0<b_j\le +\infty$. Suppose $m_j$ are defined via solutions $\theta_j(z,x)$, $\phi_j(z,x)$ which are entire functions in $z$ of order less than $1$.

If for some $c\in (0,\min\{b_1,b_2\})$ there is a real entire function $g$ 
such that for every $\eps>0$
\be\label{eq:m=m}
m_1(z)-m_2(z)=g(z) +O(\E^{-2(c-\eps)|\im(\sqrt{z})|}) 
\ee 
as $z\to \infty$ along some nonreal ray,  
then $l_1=l_2$ and $q_1(x)=q_2(x)$ for a.e. $x\in (0,c)$.
\end{theorem}

\begin{proof}
Without loss of generality we can assume that $\theta_j$ are Frobenius solutions such that the $m$-functions $m_j$ are generalized Nevanlinna functions belonging to the class $\mathcal{N}_{\kappa_j}^\infty$,
$\kappa_j=\floor{\frac{l_j}{2}+\frac{3}{4}}$. Moreover, \eqref{eq:m=m} implies that $g$ is bounded by a polynomial of degree at most $n:=\max\{\kappa_1,\kappa_2\}+1$ on two rays in $\C$ and hence, by the Phragm\'en--Lindel\"of theorem, $g$ is a polynomial of degree at most $n$. 

Therefore, without loss of generality we can absorb $g(z)$ in the $\theta_j$ with the larger $l_j$ and assume $g(z)=0$.
But then we can read off $l_j$ from the asymptotic behavior of $m_j(z)$ implying $l_1=l_2$. But this shows $\phi_1(z,x)$ and
$\phi_2(z,x)$ have the same asymptotic behavior and the result follows from \cite[Thm.~8.5]{kst2}.
\end{proof}

\begin{corollary}\label{cor:41a}
Let $H_j$, $j\in \{1,2\}$, satisfy the same assumptions as in Theorem \ref{th:uniq}. Let also the singular $m$-functions $m_j$ be defined via Frobenius solutions.  

If for some $c\in (0,\min\{b_1,b_2\})$ there is a real entire function $g$
such that \eqref{eq:m=m} holds for every $\eps>0$ 
as $z\to \infty$ along some nonreal ray,  
then $l_1=l_2$ and $q_1(x)=q_2(x)$ for a.a. $x\in (0,c)$.
\end{corollary}

\begin{proof}
It suffices to notice that according to Remark \ref{rem:2.02} Frobenius solutions are of order 1/2 and then to apply Theorem \ref{th:uniq}.
\end{proof}

Since $m(z)$ is determined by its spectral measure $\rho(\lam)$ up to an entire function (\cite[Thm.~4.1]{kst2}) we also obtain:

\begin{corollary}\label{cor:41}
Let $H_j$, $j\in \{1,2\}$, satisfy the same assumptions as in Theorem \ref{th:uniq}. Let also $m_1$ and $m_2$ be some singular $m$-functions and $\rho_1$, $\rho_2$ be the corresponding spectral functions. 

If there is a real entire function $g$ such that
\be
m_1(z)=m_2(z)+g(z),
\ee
or equivalently 
$\rho_1=\rho_2$, then $l_1=l_2$, $b_1=b_2$ and $q_1(x)=q_2(x)$ for a.e. $x\in (0,b_1)$.
\end{corollary}

\begin{proof}
Without loss of generality we can assume that $m_j(z)$ are defined via Frobenius solutions (this will only change $g(z)$).
But then $m_j(z)$ are generalized Nevanlinna functions and $g(z)$ can be at most a polynomial. Moreover,
we can even absorb $g(z)$ in the $\theta_j$ with the larger $l_j$ and assume $g(z)=0$ without loss of generality.
Hence the result follows from the previous corollary.
\end{proof}

Using a different approach based on the theory of de Branges spaces, Corollary \ref{cor:41}
was established by Eckhardt \cite{eck}.

As another consequence we also obtain a generalization of item (ii) of Theorem 2.8  from \cite{kst}.

\begin{corollary}\label{cor:42}
Suppose $H$ has purely discrete spectrum $\sig(H)=\{\lam_n \}$. Then the eigenvalues
$\lam_n$ together with the norming constants
\be
\gamma_n^{-1} = \int_0^b \phi(\lam_n,x)^2 dx
\ee
uniquely determine $q$ and $l$.
\end{corollary}

\begin{proof}
This is immediate since in this case the spectrum is purely discrete and the spectral measure is
uniquely determined by its jumps together with the jump heights $\rho(\{\lam_n\}) = \gamma_n$.
\end{proof}

\section{Quantum scattering theory}\label{sec:QST}

In this section we want to look at the case where in addition to \eqref{q:hyp} the potential has the form
\be\label{q:hyp2}
q(x) = \frac{\gam}{x} + \ti{q}(x), \qquad \ti{q}(x) \in L^1(1,\infty), \ \gam\in\R.
\ee

\subsection{The Jost function}
Recall that by Weidmann's theorem (\cite[Thm.~9.38]{tschroe}) the spectrum of $H$ is purely absolutely
continuous on $(0,\infty)$ with an at most countable number of eigenvalues $\lam_n \in (-\infty,0]$.

Then it is easy to show (cf.\ \cite{ko}) that for $k\ne 0$ there exists a unique so called {\em Jost solution} of $\tau f  = z f$
satisfying the asymptotic normalization
\be\label{eq:jost_s}
f(k,x) = \E^{\I k x - \frac{\I \gam}{2 k} \log(x)}(1 + o(1))
\ee
as $x\to \infty$. Here $k=\sqrt{z}$ with the branch cut along the positive real axis such that $0\le \arg(k) < 2\pi$.
The Jost solution is analytic in the upper half plane and can be continuously extended to the real axis
away from $k=0$. We can extend it to the lower half plane by setting $f(k,x) =f(-k,x)= f(k^*,x)^*$ for $\im(k)<0$.

Its derivative satisfies
\be
f'(k,x) = \left(\I k - \frac{\I \gam}{2 k x}\right) \E^{\I k x - \frac{\I \gam}{2 k} \log(x)}(1 + o(1)).
\ee
For $k\in\R\setminus\{0\}$ we obtain two solutions $f(k,x)$ and $f(-k,x)=f(k,x)^*$ of the same equation
whose Wronskian is given by
\be\label{eq:wrfkpm}
W(f(-k,.),f(k,.))= 2\I k.
\ee
{\em The Jost function} is defined as
\be\label{eq:jost_f}
f(k) = W(f(k,.),\phi(k^2,.))
\ee
and we also set
\be
g(k) = W(f(k,.),\theta(k^2,.))
\ee
such that
\be\label{eq:5.7}
f(k,x) = f(k) \theta(k^2,x) - g(k) \phi(k^2,x) = f(k) \psi(k^2,x).
\ee
In particular, the Weyl $m$-function \eqref{eq:s_m} is given by
\be
m(k^2) = -\frac{g(k)}{f(k)},\quad k\in\C_+.
\ee
Note that both $f(k)$ and $g(k)$ are analytic in the upper half plane and $f(k)$ has simple zeros at
$ \I \kappa_n = \sqrt{\lam_n}\in\C_+$. 

Since $f(k,x)^*=f(-k,x)$ for $k\in\R\setminus\{0\}$, we obtain $f(k)^*=f(-k)$, $g(k)^*=g(-k)$. Moreover,
\eqref{eq:wrfkpm} shows
\be\label{eq:phif}
\phi(k^2,x) = \frac{f(-k)}{2\I k} f(k,x) - \frac{f(k)}{2\I k} f(-k,x), \qquad k\in\R\setminus\{0\},
\ee
and by \eqref{eq:5.7} we get
\be
2\I\, \im(f(k) g(k)^*)= f(k)g(k)^* -  f(k)^* g(k) = W(f(-k,.),f(k,.))= 2\I k.
\ee
Hence
\be
\im(m(k^2)) = - \frac{\im\big(f(k)^*g(k)\big)}{|f(k)|^2} = \frac{k}{|f(k)|^2}, \qquad k\in\R\setminus\{0\},
\ee
implying
\be\label{eq:rho}
d\rho(\lam) = \chi_{(0,\infty)}(\lam) \frac{\sqrt{\lam}}{\pi|f(\sqrt{\lam})|^2} d\lam + \sum_n \gam_n d\theta(\lam-\lam_n),
\ee
where
\be\label{def:gam}
\gam_n = \left( \int_0^\infty \phi(\lam_n,x)^2 dx\right)^{-1}
\ee
are the usual norming constants. Since $-\gam_n$ equals the residue of $m(z)$ at $\lam_n$ we obtain
\be
\dot{f}(\I\kappa_n) = -2\I\kappa_n\frac{g(\I\kappa_n)}{\gam_n}, \qquad f(\I\kappa_n,x)= g(\I\kappa_n) \phi(\lam_n,x) .
\ee
By Theorem \ref{th:main} and \eqref{eq:rho}, 
\be
|f(k)|=\frac{|k|^{-l}}{C_l}(1+o(1)),\quad k\to \infty.
\ee

As an immediate consequence of Corollary~\ref{cor:41} we obtain

\begin{theorem}\label{th:4.1}
Suppose the potential $q$ is locally integrable and satisfies \eqref{q:hyp} near $0$ and \eqref{q:hyp2} near $\infty$.
Then the absolute value of the Jost function \eqref{eq:jost_f} on the real line together with the norming constants \eqref{def:gam} and the eigenvalues
determines the value of the angular momentum $l$ and the potential $q$ uniquely. 
\end{theorem}

\subsection{The phase shift}
In this subsection, we assume that the potential $q$ belongs to {\em the Marchenko class}, i.e., in addition to \eqref{q:hyp}, $q$ also satisfies 
\be\label{eq:q_mar}
\int_1^\infty x |\ti q(x)|dx<\infty,\quad 
\ti{q}(x)=\begin{cases}
q(x), & l>-\frac{1}{2},\\
\log(x)q(x), & l=-\frac{1}{2}.
\end{cases}
\ee
The Jost solution of $\tau f  = k^2 f$ is defined to satisfy the following asymptotic normalization:
\be\label{eq:jost_sB}
f(k,x) = \E^{\I (k x - \frac{l\pi}{2})}(1 + o(1))
\ee
as $x\to \infty$. Again $k=\sqrt{z}$ with the branch cut along the positive real axis such that $0\le \arg(k) < \pi$.
By the previous considerations, the Jost solution is analytic in the upper half plane and can be continuously extended to the real axis
away from $k=0$. We can extend it to the lower half plane by setting $f(k,x) =f(-k,x)= f(k^*,x)^*$ for $\im(k)<0$.

The Jost function $f(.)$ is defined by \eqref{eq:jost_f}. Notice that $f(-k)=f(k)^*$ for $k\in\R\setminus\{0\}$. 
Set $f(k)=|f(k)|\E^{-\I\delta(k)}$, $k\in\R\setminus\{0\}$. Note that $\delta(-k)=-\delta(k)$ and hence, by \eqref{eq:phif}, we obtain 
\be\label{eq:phase}
\phi(k^2,x) = \frac{|f(k)|}{k} \sin\left( kx - \frac{l\pi}{2} + \delta(k)\right) + o(1), \quad x\to\infty, \: 0\le \arg(k) < \pi.
\ee
The function $\delta:\R\to\R$ is called {\em the phase shift}
and {\em the scattering matrix} is defined by
\be
S(k) = \frac{f(-k)}{f(k)} = \E^{2\I \delta(k)}, \qquad k\in\R\setminus\{0\}.
\ee

\begin{remark}
In the case $q(x)=0$ we have
\begin{eqnarray}
& \phi_l(z,x) = \frac{z^{-\frac{2l+1}{4}}}{C_l} \sqrt{\frac{\pi x}{2}} J_{l+\frac{1}{2}}(\sqrt{z} x),\\
& f_l(k,x)=\I\sqrt{\frac{\pi xk}{2}}H_{l+1/2}^{(1)}(kx).
\end{eqnarray}
Therefore, 
\be
f_l(k)= \frac{k^{-l}}{C_l}, \quad 0\le \arg(k) < \pi,
\ee
and hence
\be
S_l(k)=1,\quad 
\delta_l(k)= 0,\quad k\in\R\setminus\{0\}.
\ee
\end{remark}

%

Next we are going to show that $f(k)$ is uniquely determined by its phase $\delta(k)$. 
Firstly, observe that the Jost solution \eqref{eq:jost_sB} has the form 
\be\label{eq:mod_jost}
f(k,x)=C_l^{-1}k^{-l}\ti\psi(k,x) ,\quad \im\, k\ge 0,
\ee
where $\ti\psi$ is given by \eqref{psi.a1}. 
Therefore, by Lemma \ref{lem:b.5}, the function
\be
F(k)=C_lk^l f(k)
\ee
is an analytic function in $\im\, k> 0$ satisfying \eqref{eq:F=1}. By \eqref{Flcrit} it extends continuously to the boundary if $l>-\frac{1}{2}$ and the
same is true in the critical case $l=-\frac{1}{2}$ except for a possible logarithmic singularity at $0$.

Moreover, \eqref{eq:intr_F} implies
\be\label{eq:tid}
\int_\R \frac{\delta(k)}{1+|k|}dk<\infty.
\ee
Indeed, $\delta(.)$ is bounded on $\R$ and, by \eqref{eq:F=1}, near infinity it behaves like $\im\, F$. Hence combining
\[
\im\, \psi_l(k^2,x)=\phi_l(k^2,x)\im\, m_l(k^2)=\frac{C_l^2}{\pi}k^{2l+1}\phi_l(k^2,x).
\]
with \eqref{estphil}, \eqref{eq:intr_F} and Fubini's theorem, we obtain the claimed integrability
\[
\int_1^\infty \frac{\im F(k)}{k}dk\le C\int_0^\infty x|q(x)|\int_1^\infty\frac{(kx)^{2l}x}{(1+kx)^{2l+2}}dk\, dx\le \ti{C}\int_0^\infty x|\ti{q}(x)|dx.
\]
Since $q$ satisfies \eqref{q:hyp} and \eqref{eq:q_mar}, the Bargmann bound (see, e.g., \cite{cs}) implies that $H$ only has
a finite number $N$ of negative eigenvalues $\{-\kappa_n^2\}_{n=1}^N$ and, if $l>\frac{1}{2}$, a possible eigenvalue at $0$.
Using standard techniques (c.f.\ \cite[\S 4]{n60}), it can be shown that $F(.)$ has simple zeros at $\I\kappa_n$
and at most a second order zero at $0$. Therefore,
\[
\im \log\left(\frac{C_l}{k^{-l}}\prod_{n=1}^N\left( 1+ \frac{\kappa_n^2}{k^2}\right)^{-1} f(k)\right)
\]
is in the Hardy space $H^\infty(\C_+)$ and (cf.\ \cite{kos}) we arrive at the following representation of the Jost function
\be\label{eq:intr_f}
f(k) = \frac{k^{-l}}{C_l}\prod_{n=1}^N \left( 1+ \frac{\kappa_n^2}{k^2}\right) \exp\left(  - \frac{1}{\pi} \int_\R \frac{\delta(t)}{t-k} dt \right).
\ee
Now Theorem \ref{th:4.1} immediately implies the following result.

\begin{theorem}\label{th:5.3}
Suppose the potential $q$ is locally integrable and satisfies \eqref{q:hyp} near $0$ and \eqref{eq:q_mar} near $\infty$.
Then the phase shift \eqref{eq:phase} together with the value of the angular momentum $l$, the norming constants, and the eigenvalues
determine the potential $q$ uniquely. 
\end{theorem}


\appendix

\section{Asymptotic estimates for $m$-functions}\label{app:A}

In this appendix we collect some required results from \cite{Ben}.
Consider the Sturm--Liouville differential expression 
\[
-y''+q(x)y=z\, r(x)y,\quad x\in (0,b),
\]
where $q,r\in L^1_{loc}(0,b)$ are real valued and $r>0$ a.e. on $(0,b)$.
We assume that the endpoint $x=0$ is regular, i.e., $q,r \in L^1(0,c)$ for any $c\in (0,b)$.
Let $m(\cdot)$ be the $m$-functions corresponding to the Neumann boundary condition at $x=0$
such that
\be
s(z,x) + m(z) c(z,x)
\ee
is square integrable near $b$ and satisfies a given boundary condition at $b$ in the limit circle case.

Define the following function
\be\label{eq:fR}
R(x):=\int_0^x r(t)dt
\ee
and note that $R \in W^{1,1}(0,c)$ is positive and strictly increasing on $(0,b)$.

\begin{definition}
We will say that $R$ has limit order $\alpha \in [0,\infty)$ at $x=0$ if for positive $s$
\be\label{eq:lo}
\lim_{x\to 0}\frac{R(sx)}{R(x)}=s^\alpha.
\ee
$R$ has a limit order $\infty$ at $x=0$ if  the limit in \eqref{eq:lo} equals $\infty$ for all $s>1$.
\end{definition}

Define also the function $f$ as the inverse of 
\[
F(x)=\frac{1}{xR(x)}.
\]
 Note that $f(y)\to 0$ as $y\to +\infty$.
The following result is a particular case of Theorem~4.1 from \cite{Ben} (see also \cite[Thm.~2]{kas}).

\begin{theorem}[\cite{Ben, kas}]\label{th:ben}
Assume that $R$ has limit order $\alpha$ at $x=0$. Then any $m$-function corresponding to the Neumann condition at $x=0$ satisfies
\be\label{eq:m_asymp}
m(\mu\rho)= \frac{K_\nu}{(-\mu)^\nu}f(\rho)(1+o(1)) \quad \text{as}\quad \rho\to\infty,
\ee
where
\be\label{eq:nuC}
\nu=\frac{1}{1+\alpha},\quad K_\nu=\frac{\nu^{1-\nu}\Gamma(\nu)}{(1-\nu)^\nu\Gamma(1-\nu)}.
\ee
The estimate holds uniformly for $\mu$ in any compact set of $\C_+$. Also, $K_\nu=1$ for $\nu\in\{0,1\}$.
\end{theorem}

Since it will in general not be possible to compute $f$ explicitly, we will rely on the following known fact which follows
upon combining Theorem~1 and Theorem~3 from \cite{dtj}:

\begin{lemma}[\cite{dtj}]\label{lem:a.asymp}
Let $R_0$, $R$ be two functions as in \eqref{eq:fR} and define corresponding functions $f_0$, $f$ as above.
Suppose $R_0$ has a limit order $\alpha\in(0,+\infty]$ and $\lim_{x\to 0} \frac{R(x)}{R_0(x)}=1$. Then $\lim_{y\to +\infty} \frac{f(y)}{f_0(y)}=1$.
\end{lemma}

\begin{remark}
Theorem \ref{th:ben} was first established by I.S. Kac \cite{K73} and Y. Kasahara \cite{kas} for the case $\alpha\in (0,\infty)$.
The current form of Theorem \ref{th:ben} was found by C. Bennewitz \cite{Ben}. Moreover, the converse statement has been
established independently in \cite{Ben} and \cite{kas}. 
\end{remark}

\section{Some estimates for the spherical Schr\"odinger equation}\label{app:B}

In this appendix we want to describe some properties of the solutions of the spherical Schr\"odinger equation
which are crucial for Section~\ref{sec:QST}. 

The first two lemmas contain estimates for the Green function 
\be \label{a2}
G_l(z,x,y) = \phi_l(z,x) \theta_l(z,y) - \phi_l(z,y) \theta_l(z,x)
\ee
and the regular solution $\phi(z,x)$ (see, e.g., \cite[Lems.~2.2, A.1, A.2]{kst}).  Here 
\be\label{defphil}
\phi_l(z,x) = C_l^{-1}z^{-\frac{2l+1}{4}} \sqrt{\frac{\pi x}{2}} J_{l+\frac{1}{2}}(\sqrt{z} x),
\ee
and
\be\label{defthetal}
\theta_l(z,x) = -C_lz^{\frac{2l+1}{4}} \sqrt{\frac{\pi x}{2}} \begin{cases}
\frac{-1}{\sin((l+\frac{1}{2})\pi)} J_{-l-\frac{1}{2}}(\sqrt{z} x), & {l+\frac{1}{2}} \in \R_+\setminus \N_0,\\
Y_{l+\frac{1}{2}}(\sqrt{z} x) -\frac{1}{\pi}\log(z) J_{l+\frac{1}{2}}(\sqrt{z} x), & {l+\frac{1}{2}} \in\N_0,\end{cases}
\ee
where $J_{l+1/2}$ and $Y_{l+1/2}$ are the usual Bessel and Neumann functions (see \cite[Chap.~10]{dlmf}).
We will abbreviate $k=\sqrt{z}$, $0\le\arg(z)<2\pi$.

\begin{lemma}[\cite{kst}] \label{lem:b.1}
For $l>-\frac{1}{2}$ the following estimates hold:
\be\label{estphil}
|\phi_l(k^2,x)| \leq C \left(\frac{x}{1+ |k| x}\right)^{l+1} \E^{|\im\, k| x},
\ee
\be \label{estGl}
|G_l(k^2,x,y)| \leq C \left(\frac{x}{1+ |k| x}\right)^{l+1} \left(\frac{1+ |k| y}{y}\right)^l \E^{|\im\, k| (x-y)}, \quad y \leq x,
\ee
and
\be
\label{estGlp}
\left|\frac{\partial}{\partial x} G_l(k^2,x,y)\right| \leq C \left(\frac{x}{1+ |k| x}\right)^l \left(\frac{1+ |k| y}{y}\right)^l \E^{|\im\, k| (x-y)}, \quad y \leq x.
\ee
For the case $l=-\frac{1}{2}$  formula \eqref{estphil} remains valid and one has to replace \eqref{estGl} and \eqref{estGlp} by
\begin{align}\label{a8}
|G_{-\frac{1}{2}}(k^2,x,y)| \leq C \left(\frac{x}{1+ |k| x}\cdot\frac{y}{1+ |k| y}\right)^{\frac{1}{2}}
\E^{|\im\, k| (x-y)}\big(1-\log(\frac{|k|y}{1+|k|y})\big), \quad y \leq x,\\
\left|\frac{\partial}{\partial x}G_{-\frac{1}{2}}(k^2,x,y)\right| \leq C \left(\frac{y+ |k| xy}{x+ |k| xy}\right)^{\frac{1}{2}}
\E^{|\im\, k| (x-y)}\big(1-\log(\frac{|k|y}{1+|k|y})\big), \quad y \leq x.\label{a21}
\end{align}
\end{lemma}

\begin{lemma}[\cite{kst}]\label{lem:b.2}
Assume \eqref{q:hyp} and set $\ti{q}(x)=q(x)$ if $l>-1/2$ and $\ti{q}(x)=(1-\log(\frac{x}{1+x}))q(x)$ if $l=-1/2$.
Then $\phi(z,x)$ satisfies the integral
equation
\be \label{a1}
\phi(z,x) = \phi_l(z,x) + \int_0^x G_l(z,x,y) q(y) \phi(z,y) dy.
\ee
 Moreover, $\phi$ is entire in $z$ for every $x>0$ and satisfies the estimate
\be\label{estphi}
| \phi(k^2,x) - \phi_l(k^2,x)| \leq C \left(\frac{x}{1+ |k| x}\right)^{l+1} \E^{|\im\, k| x} \int_0^x \frac{y \ti{q}(y)}{1 +|k| y} dy.
\ee
The derivative is given by
\be
\phi'(z,x) = \phi_l'(z,x) + \int_0^x \frac{\partial}{\partial x}G_l(z,x,y) q(y) \phi(z,y) dy
\ee
and satisfies the estimate
\be\label{estphi'}
| \phi'(k^2,x) - \phi_l'(k^2,x)| \leq C  \left(\frac{x}{1+ |k| x}\right)^l \E^{|\im\, k| x} \int_0^x \frac{y \ti{q}(y)}{1 +|k| y} dy.
\ee
\end{lemma}

Next we need some estimates for the Weyl solution $\psi(z,x)$. We begin with some basic properties of the unperturbed Bessel equation.

\begin{lemma}\label{lem:b.3}
Let $\psi_l(z,x)$ be the Weyl solution of the unperturbed Bessel equation,
\begin{align}\nn
\psi_l(z,x) &=\theta_l(z,x)+m_l(z,x)\phi_l(z,x)=\I C_l(\I \sqrt{-z})^{l+\frac{1}{2}}\sqrt{\frac{\pi x}{2}}H^{(1)}_{l+\frac{1}{2}}(\I \sqrt{-z}x)\\
&= \I C_l k^{l+\frac{1}{2}}\sqrt{\frac{\pi x}{2}}H^{(1)}_{l+\frac{1}{2}}(k x)
\end{align}
where $m_l$ is given by \eqref{eq:II.11} and $H_{l+\frac{1}{2}}^{(1)}$
is the Hankel function of the first kind.

 If $l>-1/2$, then $\psi_l(k^2,x)$ is analytic in $\im\, k>0$, continuous in $\im\, k\ge 0$ and  
\be\label{est:psi_lA}
|\psi_l(k^2,x)|\le C\left(\frac{x}{1+ |k| x}\right)^{-l} \E^{-|\im\, k| x}.
\ee

If $l=-1/2$, then  $\psi_{-\frac{1}{2}}(k^2,x)$ is analytic in $\im\, k>0$, continuous in $\overline{\C}_+\setminus\{0\}$ and  
\be\label{est:psi_lB}
|\psi_{-\frac{1}{2}}(k^2,x)|\le C\left(\frac{x}{1+ |k| x}\right)^{\frac{1}{2}}\left(1-\log\Big(\frac{|k|x}{1+|k|x}\Big)\right) \E^{-|\im\, k| x}.
\ee
Moreover,
\be
\sqrt{k}\psi_{-\frac{1}{2}}(k^2,x)\sim -\sqrt{kx}\log(kx),\quad kx\to 0.
\ee
\end{lemma}

Using the standard technique (see, e.g., \cite[Chap.~I.5]{cs}), one can prove the following

\begin{lemma}\label{lem:b.4}
Assume \eqref{eq:q_mar}.
Then  there is a solution $\ti\psi(z,x)$ satisfying the integral
equation
\be \label{psi.a1}
\ti\psi(z,x) = \psi_l(z,x) - \int_x^\infty G_l(z,x,y) q(y) \ti\psi(z,y) dy.
\ee

Moreover, if $l>-1/2$, then $\ti\psi(k^2,x)$ is analytic in $\im\, k>0$, continuous in $\im\, k\ge 0$ and  
it satisfies the estimate
\be\label{estpsi}
| \ti\psi(k^2,x) - \psi_l(k^2,x)| \leq C \left(\frac{x}{1+ |k| x}\right)^{-l} \E^{-|\im\, k|\, x} \int_x^\infty \frac{y q(y)}{1 +|k| y} dy.
\ee

If $l=-1/2$, then  $\ti\psi(k^2,x)$ is analytic in $\im\, k>0$, continuous in $\overline{\C}_+\setminus\{0\}$ and 
\be
| \ti\psi(k^2,x) - \psi_{-\frac{1}{2}}(k^2,x)| \leq C \left(\frac{x}{1+ |k| x}\right)^{\frac{1}{2}}\left(1-\log\Big(\frac{|k|x}{1+|k|x}\Big)\right) \E^{-|\im\, k|\, x} \int_x^\infty \frac{y \ti{q}(y)}{1 +|k| y} dy.
\ee
Moreover, 
\be
|\sqrt{k}\ti\psi(k^2,x)|=O\big(-\sqrt{|k|x}\log(|k|x)\big),\quad |k|x\to 0.
\ee

The derivative is given by
\be
\ti\psi'(z,x) = \psi_l'(z,x) - \int_x^\infty \frac{\partial}{\partial x}G_l(z,x,y) q(y) \ti\psi(z,y) dy
\ee
and satisfies the estimate
\be
| \ti\psi'(k^2,x) - \psi_l'(k^2,x)| \leq C  \left(\frac{x}{1+ |k| x}\right)^{-l-1} \E^{-|\im\, k|\, x} \int_x^\infty \frac{y q(y)}{1 +|k|y} dy.
\ee
if $l>-1/2$ and 
\be
| \ti\psi'(k^2,x) - \psi_{-\frac{1}{2}}'(k^2,x)| \leq C  \left(\frac{x}{1+ |k| x}\right)^{-\frac{1}{2}}\left(1-\log\Big(\frac{|k|x}{1+|k|x}\Big)\right) \E^{-|\im\, k|\, x} \int_x^\infty \frac{y \ti{q}(y)}{1 +|k|y} dy.
\ee
if $l=-1/2$.
\end{lemma}

Finally, consider the following function
\be\label{eq:a.F}
F(k):=W(\ti\psi(k^2,.),\phi(k^2,.)),\quad \im\, k\ge 0,\ k\neq 0. 
\ee

\begin{lemma}\label{lem:b.5}
Assume \eqref{q:hyp} and \eqref{eq:q_mar}. Then the  function $F$ admits the following integral representation
\be\label{eq:intr_F}
F(k)=1+\int_0^\infty \psi_l(k^2,x)\phi(k^2,x) q(x)dx = 1+\int_0^\infty \ti{\psi}(k^2,x)\phi_l(k^2,x) q(x)dx.
\ee
Moreover, $F$ is analytic in $\C_+$ and 
\be\label{eq:F=1}
F(k)=1+o(1)
\ee
as $|k|\to \infty$ in $\im\, k\ge 0$. 

If $l>-1/2$, then $F$ is continuous and bounded on $\im\, k\ge 0$. In the case $l=-1/2$    
\be\label{Flcrit}
F(k)=c\, \log(-k^2) + \ti{F}(k),\quad k\in\R\setminus\{0\}
\ee
where $c\in\R$ and $\ti{F}$ is bounded continuous, and also 
\be\label{eq:b.24}
|F(k)|=O(\log(k))\ \ \ \text{as}\ \ k\to 0\ \ \ \text{in}\ \ \im\, k\ge 0.
\ee
 \end{lemma}

\begin{proof}
To prove the integral representations \eqref{eq:intr_F}, we need to replace $\phi$ and $\psi$ in \eqref{eq:a.F} by \eqref{a1} and \eqref{psi.a1}, respectively, use the asymptotic estimates for $\phi$, $\psi$ and $G_l$, and then take the limits $x\to +\infty$ and $x\to 0$. 

To prove the second claim, we observe that by \eqref{estphil} and \eqref{est:psi_lA}
\begin{align}\label{eq:b.22}
|\psi_l(k^2,x)\phi(k^2,x)|&\le C\frac{x}{1+|k|x},\quad l>-\frac{1}{2},\\
|\psi_{-\frac{1}{2}}(k^2,x)\phi(k^2,x)|&\le C\frac{x}{1+|k|x}\left(1-\log\Big(\frac{|k|x}{1+|k|x}\Big)\right),\quad l=-\frac{1}{2},\label{eq:b.23}
\end{align}
which immediately implies \eqref{eq:F=1}.

Moreover, if $l>-1/2$, then \eqref{eq:b.22} shows that in this case $F$ is analytic in $\C_+$ and continuous in $\overline{\C}_+$.
If $l=-1/2$, then \eqref{eq:b.23} implies analyticity in $\C_+$ and also the estimate \eqref{eq:b.24}.
Finally, notice that 
\[
F(k)=1+\int_0^{+\infty} \theta_l(k^2,x)\phi(k^2,x) q(x)dx + m_l(k^2)\int_0^{+\infty} \phi_l(k^2,x)\phi(k^2,x) q(x)dx
\]
and the integrals converge for $k\in \R$. This implies \eqref{Flcrit}.
\end{proof}

\begin{remark}
By choosing $q$ to be a characteristic function in \eqref{eq:intr_F}, one sees that $c$ in \eqref{Flcrit} is nonzero in general.
\end{remark}

\noindent
{\bf Acknowledgments.}
We thank Fritz Gesztesy and Mark Malamud for several helpful discussions and hints with respect to the literature. 

\end{document}